\documentclass[reqno, 12pt]{amsart}
\usepackage[latin1]{inputenc}
\usepackage{amsfonts}
\usepackage{amsmath}
\usepackage{amssymb}
\usepackage{amsthm}
\usepackage{fontenc}
\usepackage{graphicx}
\usepackage[margin=1.45in]{geometry}

\newtheorem{thm}{Theorem}

\newtheorem{prop}[thm]{Proposition}

\newtheorem{cor}[thm]{Corollary}
\newtheorem{df}{Definition}

\title{Equivariant Equipartitions: Ham Sandwich Theorems for Finite Subgroups of Spheres}

\author{Steven Simon}

\begin{document}

\maketitle

\begin{abstract} 

Equivariant ``Ham Sandwich" Theorems are obtained for the finite subgroups $G$ of the unit spheres $S(\mathbb{F})$ in the classical algebras $\mathbb{F}=\mathbb{R},\mathbb{C},$ and $\mathbb{H}$.  Given any $n$ $\mathbb{F}$-valued mass distributions on $\mathbb{F}^n$, it is shown that there exists a $G$-equivariant decomposition of $\mathbb{F}^n$ into $|G|$ regular convex fundamental regions which ``$G$-equipartition" each of the $n$ measures, as realized by the vanishing of the ``$G$-averages" of these regions' measures. Applications for real measures follow, among them that any $n$ signed mass distributions on $\mathbb{R}^{(p-1)n}$ can be  equipartitioned by a single regular $p$-fan when $p$ a prime number.

\end{abstract}

\section{Introduction}

The familiar Ham Sandwich Theorem states that any $n$ mass distributions on $n$-dimensional Euclidian space can be bisected by a single hyperplane. Reinterpreting the theorem as a $\mathbb{Z}_2$-symmetry statement for measures, we provide analogous ``$G$-equipartition" theorems for measures by other finite groups $G$. In each case, we find a natural partition of space into $|G|$ regular convex fundamental regions whose geometric $G$-symmetry are realized as a simultaneous $G$-symmetry of the mass-distributions considered. 
	
\subsection{A $\mathbb{Z}_2$-Reformulation of the Ham Sandwich Theorem}
	
	Let $\mathbb{Z}_2\cong\{\pm1\}$ denote the cyclic group of order two. Each pair $\{(\mathbf{a},b), (-\mathbf{a},-b)\}$ of antipodal points in $S^{n-1} \times \mathbb{R}$ determines a unique pair of halfspaces $H^+:=H^+(\mathbf{a},b)=\{\mathbf{u}\in\mathbb{R}^n\mid \langle\mathbf{u},\mathbf{a}\rangle_{\mathbb{R}}\geq b\}$ and $H^-:=H^-(\mathbf{a},b)=\{\mathbf{u}\mid \langle\mathbf{u},\mathbf{a}\rangle_{\mathbb{R}}\leq b\}$ on which $\mathbb{Z}_2$ acts freely and transitively by the action $\cdot$ generated by  reflection about their common hyperplane $H_{\mathbb{R}}:=H_{\mathbb{R}}(\mathbf{a},b)=\{\mathbf{u}\mid \langle \mathbf{u},\mathbf{a}\rangle_{\mathbb{R}}=b\}$. If $\mu$ is a mass distribution on $\mathbb{R}^n$, i.e., a finite Borel measure on $\mathbb{R}^n$ for which each hyperplane has measure zero (or more generally, a signed mass distribution - see Definition 1 below), the ``$\mathbb{Z}_2$-average" \begin{equation} \mu(H^+) - \mu(H^-) = 1\cdot \mu(H^+\cdot 1) + (-1)\cdot \mu(H^+\cdot (-1))\end{equation} evaluates the $\mathbb{Z}_2$-symmetry of the measures of the $\mathbb{Z}_2$-symmetric half-spaces $\{H^+,H^-\}$ with respect to the free orthogonal $\mathbb{Z}_2$-action $\cdot$ on $\mathbb{R}$ given by restricting multiplication to $\{\pm1\}$. The vanishing of this average can rightly be called a $\mathbb{Z}_2$-equipartition of $\mu$, and the Ham Sandwich Theorem states that any $n$ mass distributions on $\mathbb{R}^n$ can be $\mathbb{Z}_2$-equipartitioned by a single pair of half-spaces. 	

\subsection{ $G$-Ham Sandwich Theorems} 

	As $\{\pm1\}=S^0$ is the unit sphere in $\mathbb{R}$, it is natural to ask for generalizations of the Ham Sandwich Theorem to the finite subgroups of the unit circle $S^1\subseteq\mathbb{C}$ and the 3-sphere $S^3\subseteq\mathbb{H}$, the unit spheres of the other classical (skew-) fields. Owing to the use of complex and quaternionic multiplication, we shall consider the more general complex or quaternion-valued mass distributions of Definition 1 below.

	Let $G$ be a non-zero finite subgroup of $S(\mathbb{F})$, $\mathbb{F}=\mathbb{R},\mathbb{C},$ or $\mathbb{H}$. We will show that $G$ determines a natural family  of (``$G$-Voronoi") partitions $\{\mathcal{R}_g:=\mathcal{R}_g(\mathbf{a},b)\}_{g\in G}$ of $\mathbb{F}^n$ into regular convex fundamental regions parametrized by the $G$-orbits $\{(g\mathbf{a},gb)\}_{g\in G}$ of $S(\mathbb{F}^n)\times\mathbb{F}$. Each collection of regions will be centered about a $\mathbb{F}$-hyperplane $H_{\mathbb{F}}:=H_{\mathbb{F}}(\mathbf{a},b)$, and $G$ will act freely and transitively on each partition by $\mathbb{F}$-linear isometries about $H_{\mathbb{F}}$ - orthogonally when $\mathbb{F}=\mathbb{R}$, unitarily when $\mathbb{F}=\mathbb{C}$, and symplectically when $\mathbb{F}=\mathbb{H}$. On the other hand, the free $\mathbb{F}$-isometric $G$-actions on $\mathbb{F}$ are obtained by setting $g\cdot u=\varphi(g)u$ for each $u\in\mathbb{F}$ and $g\in G$ when $\varphi\in Aut(G)$ (a linear action $\cdot$ is free provided $g\cdot u =u$ iff $g=1$ or $u=0$). Given a $\mathbb{F}$-valued mass distribution $\mu$ on $\mathbb{F}^n$ and a $G$-Voronoi partition $\{\mathcal{R}_g\}_{g\in G}$ of $\mathbb{F}^n$, the ``$(G,\varphi)$-average" \begin{equation} \sum_{g\in G} \varphi(g)^{-1}\mu(\mathcal{R}_g) = \sum_{g\in G} g^{-1}\cdot \mu(\mathcal{R}_1\cdot g)\in\mathbb{F}\end{equation} evaluates the $G$-symmetry of the measures of the fundamental $G$-regions with respect to the $G$-action on $\mathbb{F}$ afforded by the automorphism $\varphi$. We shall therefore say that a $G$-Voronoi partition $(G,\varphi)$-equipartitions $\mu$ if this average is zero.\\
		
	The main result of this paper states that any $n$ $\mathbb{F}$-valued measures on $\mathbb{F}^n$ can be $G$-equipartitioned simultaneously, thereby realizing the $G$-symmetry on $\mathbb{F}^n$ as a corresponding $G$-symmetry of its finite measures. The original Ham Sandwich Theorem is recovered when $\mathbb{F}=\mathbb{R}$, since the $S(\mathbb{R})$-Voronoi partitions of $\mathbb{R}^n$ are the half-space decompositions (see section 2) and the only automorphism of $S(\mathbb{R})$ is the identity map. 

\begin{thm} $G$-Ham Sandwich Theorem\\ 
Let $\mathbb{F}=\mathbb{R},\mathbb{C},$ or $\mathbb{H}$ and let $G$ be a non-zero finite subgroup of $S(\mathbb{F})$. Given any $n$ $\mathbb{F}$-valued mass distributions $\mu_1,\ldots \mu_n$ on $\mathbb{F}^n$ and any $n$ automorphisms $\varphi_1,\ldots,\varphi_n\in Aut(G)$, there exists a $G$-Voronoi partition $\{\mathcal{R}_g\}_{g\in G}$  which $(G,\varphi_i)$-equipartitions each $\mu_i:$ 
\begin{equation} \sum_{g\in G} \varphi_i(g)^{-1}\mu_i(\mathcal{R}_g)=0 \end{equation} for each  $1\leq i \leq n$. \end{thm}

\noindent One also has the following generalization of Theorem 1 for cosets of a subgroup $H\leq G$:

\begin{thm} Let $H$ be a non-zero index $k$ subgroup of a finite group $G\leq S(\mathbb{F})$. Given any $n$ $\mathbb{F}$-valued mass distributions $\mu_1,\ldots, \mu_n$ on $\mathbb{F}^{kn}$ and any $kn$ automorphisms $\varphi_{\ell_1}, \ldots, \varphi_{\ell_n} \in Aut(H)$, $1\leq \ell \leq k$, there exists a $G$-Voronoi partition $\{\mathcal{R}_g\}_{g\in G}$ such that 

\begin{equation} \sum_{h\in H} \varphi_{\ell_i}(h)^{-1}\mu_i(\mathcal{R}_{g_\ell h})=0\end{equation} for each $1\leq \ell \leq k$ and $1\leq i \leq n$, where $g_1,\ldots, g_k$ are representatives of the cosets $g_1H,\ldots, g_kH$ of $H$. \end{thm}

	After describing $G$-Voronoi partitions in section 2, we shall be primarily concerned with the complex cases of Theorems 1 and 2 (section 3) and its applications for equipartitions of signed mass distributions by regular $m$-fans (section 4). The quaternionic cases are discussed in section 5, and the proofs of the two theorems are given in section 6. We conclude this paper by discussing a possible extension of Theorem 1 to groups which act freely and linearly on spheres.  First, we give a formal definition of the measures to which the above two theorems apply. 

\subsection{$\mathbb{F}$-Valued Mass Distributions on $\mathbb{F}^n$} 

\begin{df} Let $\mathbb{F} = \mathbb{R}, \mathbb{C}$, or $\mathbb{H}$, and let $\mathfrak{B}(\mathbb{F}^n)$ denote the Borel sets of $\mathbb{F}^n$. A function $\mu: \mathfrak{B}(\mathbb{F}^n) \longrightarrow \mathbb{F}$ will be called a $\mathbb{F}$-valued mass distribution provided\\
\\
$(i)$ $\mu(\emptyset) = 0$\newline
$(ii)$ If $\{E_i\}_{i=1}^{\infty}$ is a countable collection of disjoint Borel sets, then $\mu(\bigcup_{i=1}^{\infty} E_i)= \sum_{i=1}^{\infty}\mu(E_i)$, and this sum converges absolutely with respect to the Euclidian norm on $\mathbb{F}$.\newline
$(iii)$ Each hyperplane $H\subseteq \mathbb{R}^{dn}$ is a null set, i.e., $\mu(E)=0$ for each Borel set $E\subseteq H$.\newline
$(iv)$ $\mu(\mathbb{F}^n)\neq 0$\end{df}

	Conditions $(i)$ and $(ii)$ are the definition of a $\mathbb{F}$-valued Borel measure on $\mathbb{F}^n$ (see, e.g., [9]), condition $(iii)$ provides the analogous propriety condition on hyperplanes as for positive mass distributions on $\mathbb{R}^n$, while condition $(iv)$ will be needed for technical reasons in the proofs of the theorems.  When $\mathbb{F}=\mathbb{R}$, $\mu$ will be called a  signed mass distribution. Thus each $\mathbb{F}$-valued mass distribution $\mu$ on $\mathbb{F}^n$ can be expressed uniquely as $\mu = \sum_{b\in \mathcal{B}(\mathbb{F})}\mu_b b$, where each $\mu_b$ is a singed mass distribution on $\mathbb{R}^{dn}$, $d=dim(\mathbb{F})$, and $\mathcal{B}(\mathbb{F})$ denotes the standard basis for $\mathbb{F}$ - $\mathcal{B}(\mathbb{R})=\{1\}, \mathcal{B}(\mathbb{C})=\{1,i\}$, and $\mathcal{B}(\mathbb{H})=\{1,i,j,k\}$. For example, if $m$ denotes Lebesgue measure on $\mathbb{F}^n$ restricted to its Borel sets and $f=\sum_{b\in\mathcal{B}(\mathbb{F})}f_bb:\mathbb{F}^n\longrightarrow\mathbb{F}$ is a $L^1(m)$-function, then letting $\mu(E)=\int_Ef\mathrm{d}m:=\sum_{b\in\mathcal{B}(\mathbb{F})}(\int_Ef_b\mathrm{d}m)b$ for each $E\in \mathfrak{B}(\mathbb{F}^n)$ defines a $\mathbb{F}$-valued mass distribution on $\mathbb{F}^n$ provided $\int_{\mathbb{F}^n} f \mathrm{d}m\neq 0$. 
	
\section{$G$-Voronoi Partitions}

	We provide our construction of the regular $G$-partitions associated to the standard action of $G\leq S(\mathbb{F})$ on $\mathbb{F}^n$. For each $(\mathbf{a},b)\in S(\mathbb{F}^n)\times\mathbb{F}$, let $\{g\mathbf{a}\}_{g\in G} + \bar{b}\mathbf{a}$ be the translated $G$-orbit $\{g\mathbf{a}\}_{g\in G}$ of $\mathbf{a}\in S(\mathbb{F}^n)$ by  $\bar{b}\mathbf{a}$, where $\bar{b}$ denotes the $\mathbb{F}$-conjugate of $b\in\mathbb{F}$. The Voronoi partition corresponding to these $|G|$ points will be called a \emph{$G$-Voronoi partition of $\mathbb{F}^n$}. Thus each translated $G$-orbit determines the partition $\{\mathcal{R}_g(\mathbf{a},b)\}_{g\in G}$ of $\mathbb{F}^n$ into convex regions
		
	\begin{equation} \mathcal{R}_g(\mathbf{a},b) = \{\mathbf{u}\in \mathbb{F}^n \mid ||\mathbf{u}-\bar{b}\mathbf{a} - g\mathbf{a}|| \leq ||\mathbf{u} -\bar{b}\mathbf{a} - g'\mathbf{a}|| \hspace{.2cm} \forall g'\in G\} \end{equation}

	Let  $\langle \mathbf{u}, \mathbf{v} \rangle_{\mathbb{F}} =  \sum_{i=1} u_i\bar{v}_i \in\mathbb{F}$ denote the standard $\mathbb{F}$-valued inner product on $\mathbb{F}^n$, $\mathbf{u}=(u_1,\ldots, u_n), \mathbf{v} = (v_1,\ldots, v_n) \in  \mathbb{F}^n$. The $\mathcal{R}_g$ can be described compactly in terms of this product by using its $\mathbb{F}$-linearity and the Pythagorean Theorem. Namely, if $\{R_g:=\mathcal{R}_g(1,0)\}_{g\in G}$ is the Voronoi partition of $\mathbb{F}$ by $G\subseteq\mathbb{F}$, then  $\mathcal{R}_g(\mathbf{a},b) = \langle \mathbf{a}\rangle_{\mathbb{F}}^\perp + \bar{b}\mathbf{a} + R_g\mathbf{a}$, where $\langle\mathbf{a}\rangle_{\mathbb{F}}^\perp=\{\mathbf{u}\mid \langle\mathbf{u},\mathbf{a}\rangle_{\mathbb{F}}=0\}$ is the $\mathbb{F}$-complement of $\mathbf{a}$, and so 
	
	\begin{equation} \mathcal{R}_g(\mathbf{a},b) = \{\mathbf{u}\in \mathbf{F}^n \mid \langle \mathbf{u}, \mathbf{a} \rangle_{\mathbb{F}} = \bar{b} + v; v\in R_g\} \end{equation} 
	
\noindent In particular,the $\mathcal{R}_g(\mathbf{a},b)$ are centered about the $\mathbb{F}$-hyperplane 

\begin{equation} H_{\mathbb{F}}(\mathbf{a},b) = \{\mathbf{u} \in \mathbb{F}^n \mid \langle \mathbf{u}, \mathbf{a} \rangle_{\mathbb{F}} = \bar{b}\},\end{equation} 

\noindent a (real) codimension $d=dim(\mathbb{F})$ affine space. As multiplication of $\mathbb{F}$ on the right by $g$ is a $\mathbb{F}$-linear isometry (i.e., preserving the $\mathbb{F}$-valued inner product), $G$ acts freely and transitively on $\{\mathcal{R}_g(\mathbf{a},b)\}_{g\in G}$  on the right by affine $\mathbb{F}$-linear isometries which fix the $\mathbb{F}$-hyperplane $H_{\mathbb{F}}(\mathbf{a},b)$: $\mathcal{R}_{g_1}(\mathbf{a},b)\cdot g_2 = \mathcal{R}_{g_1g_2}(\mathbf{a},b)$ for each $g_1,g_2\in G$.

	By elementary properties of conjugation, (6) and (7) show that $\mathcal{R}_{g_1}(g_2\mathbf{a},g_2 b)= \mathcal{R}_{g_1g_2}(\mathbf{a},b)$ and $H_{\mathbb{F}}(g\mathbf{a},gb)=H_{\mathbb{F}}(\mathbf{a},b)$  for all $g_1,g_2\in G$. Thus (a) each translated $G$-orbit determines the same $G$-Voronoi partition  and (b) the diagonal $G$-action on $S(\mathbb{F}^n)\times\mathbb{F}$ corresponds precisely to the $G$-action on the $G$-Voronoi regions about their fixed $\mathbb{F}$-hyperplane. Thus the space $\mathcal{P}(\mathbb{F}^n; G)$ of all $G$-Voronoi partitions of $\mathbb{F}^n$ can be identified in a canonical fashion with the tautological $\mathbb{F}$-line bundle 
		\begin{equation} \mathbb{F} \hookrightarrow (S(\mathbb{F}^n)\times\mathbb{F})/G \rightarrow S(\mathbb{F}^n)/G \end{equation} over the spherical space-form $S(\mathbb{F}^n)/G$ associated to the linear $G$-action on $\mathbb{F}^n$ (see, e.g., [11]). For example, when $\mathbb{F}=\mathbb{R}$ one has $G=S^0\cong\mathbb{Z}_2$. As $R_1=[0,\infty)$ and $R_{-1}=(-\infty, 0]$, (6) and (7) show that the half-space decompositions of $\mathbb{R}^n$ are precisely  its  $S^0$-Voronoi partitions, the diagonal $S^0$-action on $S^{n-1}\times\mathbb{R}$ realizes the $S^0$-action on pairs of half-spaces, and the space $\mathcal{P}(\mathbb{R}^n;\mathbb{Z}_2)$ of all half-space decompositions is the canonical real line bundle over real projective space $\mathbb{R}P^{n-1}=S^{n-1}/\mathbb{Z}_2$.

\section{A $\mathbb{Z}_m$-Ham Sandwich Theorem for Complex Measures}

	We now turn to the complex cases of Theorems 1 and 2. The non-zero finite subgroups of $S(\mathbb{C})=S^1=\{z\in\mathbb{C}\mid |z|=1\}$ are precisely the cyclic groups $\mathbb{Z}_m$ of order $m\geq2$, realized explicitly as the $m$-th roots of unity $C_m = \{\zeta_m^k\mid 0\leq k <m\}$, $\zeta_m = e^{\frac{2\pi i}{m}}$. 
	
	The $\mathbb{Z}_m$-Voronoi partitions can be described as follows. First, note that $C_m\subseteq\mathbb{C}$ gives a Voronoi partition of $\mathbb{C}^n$ into regular $m$-sectors $\{S_k\}_{k=0}^{m-1}:=\{R_{\zeta_m^k}\}_{k=0}^{m-1}$ centered at the origin, with $\zeta_m^k\in S_k$. As each $(\mathbf{a},b)\in S^{2n-1}\times\mathbb{C}$ defines a complex hyperplane $H_{\mathbb{C}}=H_{\mathbb{C}}(\mathbf{a},b)=\{\mathbf{u}\in\mathbb{C}^n\mid \langle \mathbf{u},\mathbf{a} \rangle_{\mathbb{C}} = \bar{b}\}$  and each $\lambda\in S^1$ determines a corresponding half-hyperplane $H_\lambda(\mathbf{a},b)=\{\mathbf{u}\mid \langle\mathbf{u},\mathbf{a}\rangle_{\mathbb{C}} = \bar{b}+r\lambda\mid r\geq0\}$, (7) shows that each 
	
		 \begin{equation}\mathcal{S}_k(\mathbf{a},b):=\mathcal{R}_{\zeta_m^k}(\mathbf{a},b)=\{\mathbf{u}\mid \langle\mathbf{u},\mathbf{a}\rangle_{\mathbb{C}}=\bar{b}+ v\in S_k\}, \end{equation}  $0\leq k<m$, is the closed regular $m$-sector lying between the half-hyperplanes $H_{\zeta_m^k}(\mathbf{a},b)$ and $H_{\zeta_m^{k+1}}(\mathbf{a},b)$. Thus the $\mathbb{Z}_m$-Voronoi partitions decompose $\mathbb{C}^n$ into regular $m$-sectors $\{\mathcal{S}_k\}_{k=0}^{m-1}$, and $\mathbb{Z}_m$ acts freely and transitively on each collection of sectors by unitary rotations by multiples of $\frac{2\pi}{m}$ about their centering complex hyperplane:  $\mathcal{S}_j(\mathbf{a},b) \cdot \zeta_m^k = \mathcal{S}_{j+k}(\mathbf{a},b)$ for each $0\leq j,k < m$, where addition is taken \textit{modulo m}.\\

	As a final bit of notation, let \begin{equation} F_m=F_m(\mathbf{a},b) = \bigcup_{k=0}^{m-1} H_{\zeta_m^k}(\mathbf{a},b)\end{equation} denote the union of the $m$ half-hyperplanes composing the boundaries of the regular sectors. Thus each $F_m(\mathbf{a},b)$ is a regular ``$m$-fan" in $\mathbb{R}^{2n}$, i.e., the union of $m$ half-hyperplanes with common boundary a codimension 2-flat, each of whose successive angles is $2\pi/m$. Having a centering complex hyperplane, we shall call the $F_m$ ``complex regular $m$-fans". For each $\lambda\in S^1$, $F_m(\lambda \mathbf{a},\lambda b)$ is the regular $m$-fan $F_m(\mathbf{a},b)$ rotated by $\arg(\lambda)$ about $H_{\mathbb{C}}(\mathbf{a},b)$, so that all regular $m$-fans in $\mathbb{C}^n$ centered about a complex hyperplane are realized by this construction.\\	
	
	Now let $\mu$ be a complex-valued mass distribution on $\mathbb{C}^n$. In order to form the various $\mathbb{Z}_m$-averages of the measures of regular $m$-sectors, we note the elementary fact that (a) the automorphisms $\varphi$ of $\mathbb{Z}_m$ are precisely the homomorphisms $\varphi:C_m\longrightarrow C_m$ which send $\zeta_m$ to a primitive $m$-th root of unity $\zeta_m^r$, $1\leq r<m$ relatively prime to $m$, so that (b) the free unitary $\mathbb{Z}_m$-actions on $\mathbb{C}$ are given by rotating $\mathbb{C}$ about the origin by multiples of $\frac{2\pi r}{m}$ for these $r$: $\zeta_m^k\cdot z = \zeta_m^{kr}z$  for each $1\leq k <m$. Thus the various $\mathbb{Z}_m$-averages \begin{equation} \sum_{k=0}^{m-1}\zeta_m^{-rk}\mu(\mathcal{S}_k) = \sum_{k=0}^{m-1} \zeta_m^{-k}\cdot \mu(\mathcal{S}_0\cdot \zeta_m^k) \in\mathbb{C}\end{equation}  evaluate the unitary $\mathbb{Z}_m$-rotational symmetry of the measures of the regular $m$-sectors with respect to the free unitary $\mathbb{Z}_m$-rotational actions on $\mathbb{C}$, and Theorem 1 reads as the following $\mathbb{Z}_m$-equipartition statement:

\begin{thm} The $\mathbb{Z}_m$-Ham Sandwich Theorem\\
\noindent Let $m\geq2$. Given any complex-valued mass distributions $\mu_1,\ldots,\mu_n$ on $\mathbb{C}^n$ and any integers $1\leq r_1,\ldots, r_n<m$ relatively prime to $m$, there exists a complex regular $m$-fan $F_m$ whose regular $m$-sectors $\{\mathcal{S}_k\}_{k=0}^{m-1}$ satisfy \begin{equation}  \sum_{k=0}^{m-1}\zeta_m^{-r_ik}\mu_i(\mathcal{S}_k) = 0 \end{equation} for each $1\leq i \leq n$.\end{thm}

	For any $k$ dividing $m$, restricting the $\mathbb{Z}_m$-action on a given collection of regular $m$-sectors $\{\mathcal{S}_j\}_{j=0}^{m-1}$ to the subgroup $\mathbb{Z}_k$ gives a free and transitive $\mathbb{Z}_k$-action on each set of $\mathbb{Z}_k$-orbits $\mathcal{O}_\ell:=\{\mathcal{S}_{\ell+jt}\}_{j=0}^{k-1}$, $0\leq \ell < t=m/k$. Given a complex mass distribution, the various $\mathbb{Z}_k$-averages evaluate the $\mathbb{Z}_k$-symmetry of the measures of the regular $m$-sectors in each orbit, and Theorem 2 yields the following symmetry statement: 

\begin{thm} Suppose $k$ divides $m$, $k\geq2$, and let $t=m/k$. For any $n$ complex-mass distributions $\mu_1,\ldots, \mu_n$ on $\mathbb{C}^{tn}$ and any $tn$ integers $1\leq r_{\ell_1},\ldots, r_{\ell_n} <k$ relatively prime to $k$, $0\leq \ell <t$, there exists a complex regular $m$-fan $F_m$ whose regular $m$-sectors $\{\mathcal{S}_i\}_{i=0}^{m-1}$ satisfy \begin{equation} \sum_{j=0}^{k-1}\zeta_k^{-r_{\ell_i}j}\mu_i(\mathcal{S}_{\ell + jt})=0\end{equation} for each $0 \leq \ell < t$ and $1\leq i \leq n$. 
\end{thm}

\section{Applications to Signed Mass Distributions: Equipartitions by Regular Fans}

	We now deduce some interesting equipartition statements for real measures which follow as corollaries of Theorems 3 and 4. To begin, as each complex-valued mass distribution on $\mathbb{C}^n$ is a pair of signed mass distributions on $\mathbb{R}^{2n}$, the $\mathbb{Z}_m$-Ham Sandwich Theorem can be applied to any $2n$ signed measures $\mu_1^+,\mu_1^-, \ldots, \mu_n^+, \mu_n^-$ on $\mathbb{R}^{2n}$, thereby yielding a rotational condition on pairs of measures. When $m=2$, examination of the real and imaginary parts of (12) shows that there is a real hyperplane $H=F_2$  whose half-spaces $\mathcal{S}_0$ and $\mathcal{S}_1$ satisfy $\mu_i^\pm(\mathcal{S}_0)-\mu_i^\pm(\mathcal{S}_1)=0$ for each $1\leq i\leq n$. As $\mu_i^\pm(H)=0$,  $H$ equipartitions each $\mu_i^\pm$. Thus the $\mathbb{Z}_2$-Ham Sandwich Theorem for complex-valued mass distributions is precisely the even-dimensional Ham Sandwich theorem for signed mass distributions on $\mathbb{R}^{2n}$.\\

For odd primes, one also has  the following equipartition theorem: 

\begin{cor} Let $p$ be an odd prime number. Any $n$ signed mass distributions $\mu_1,\ldots, \mu_n$ on $\mathbb{R}^{(p-1)n}$ can be equipartitioned by a single regular $p$-fan $F_p:$ 
\begin{equation} \mu_i(\mathcal{S}_0)=\mu_i(\mathcal{S}_1)=\ldots=\mu_i(\mathcal{S}_{p-1}) = \frac{1}{p} \mu_i(\mathbb{R}^{(p-1)n}) \end{equation} for each $1\leq i \leq n$, where $\{\mathcal{S}_k\}_{k=0}^{p-1}$ are the regular $p$-sectors of $F_p$.\end{cor}

\begin{proof} We apply the $\mathbb{Z}_p$-Ham Sandwich Theorem on  $\mathbb{R}^{(p-1)n}=\mathbb{C}^{(p-1)n/2}$ to the $\frac{p-1}{2}n$ signed mass distributions $\nu_{i,k}=\mu_i$ and $\frac{p-1}{2}n$ integers $r_{i,k}=k$ relatively prime to $p$, $1\leq i \leq n$ and $1\leq k\leq (p-1)/2$.  As each $\mu_i$ is real-valued, examining real and imaginary parts in (12) shows that there exists a regular $p$-fan $F_p$ whose regular sectors satisfy
	\begin{equation} \sum_{k=0}^{p-1}\cos(\frac{2\pi kr}{p})\mu_i(\mathcal{S}_k) =  \sum_{k=0}^{p-1}\sin(\frac{2\pi kr}{p})\mu_i(\mathcal{S}_k) = 0\end{equation} for each $1\leq i \leq n$ and each $1\leq r\leq (p-1)/2$. 
	
	As $\sum_{r=1}^{\frac{p-1}{2}}\cos(\frac{2\pi rk}{p})=-\frac{1}{2}$ for each $1\leq k<p$, summing (12) over $r$ yields $\frac{p-1}{2}\mu_i(\mathcal{S}_0) -\frac{1}{2}(\mu_i(\mathcal{S}_1)+\ldots + \mu_i(\mathcal{S}_{p-1})) = 0$. On the other hand, $\sum_{k=1}^{p-1}\mu_i(\mathcal{S}_k)= \mu_i(\mathbb{R}^{(p-1)n})-\mu_i(\mathcal{S}_0)$ because $\mu_i(F_p)=0$, and therefore  $\mu_i(\mathcal{S}_0)=\frac{1}{p}\mu_i(\mathbb{R}^{(p-1)n})$ for each $1\leq i \leq n$. Multiplying (12) by $\zeta_p^{rj}$ for each $1\leq r\leq (p-1)/2$ and applying the same reasoning yields $\mu_i(\mathcal{S}_j)= \frac{1}{p}\mu_i(\mathbb{R}^{(p-1)n})$ for each $1\leq j<p$ and each $1\leq i \leq n$. \end{proof}
	
	In particular, Corollary 5 shows that any signed mass distribution on $\mathbb{R}^{p-1}$ can be equipartitioned by a regular $p$-fan. One may also ask for equipartition theorems by regular $m$-fans when $m$ is composite. Some results along these lines are given in [13],  where it shown that given any  (positive)  mass distribution on $\mathbb{R}^{p+1}$, $p$ an odd prime, one can decompose $\mathbb{R}^{p+1}$ as the union of $2p$ regular $2p$-sectors of equal measure. The interiors of the sectors may intersect non-trivially, however, so one does not necessarily have an equipartition by a regular $2p$-fan. Nonetheless, the methods used there do recover the classical result [6] that any measure in the plane can be equipartitioned by a regular 4-fan.
	
	As a final application of Theorem 3, examining real and imaginary parts in the $\mathbb{Z}_4$-equipartition formula (12) yields that for any $n$ signed mass distributions $\mu_1,\ldots, \mu_n$ on $\mathbb{R}^{2n}$, there exists a regular $4$-fan $F_4$ whose opposite regular sectors $\mathcal{S}_k$ and $\mathcal{S}_{k+2}$ always have equal measure: $\mu_i(\mathcal{S}_k)=\mu_i(\mathcal{S}_{k+2})$ for each $k=0,1$ and each $1\leq i \leq n$. As $F_4$ is the union of two orthogonal hyperplanes, we see that 

\begin{cor} Any $n$ signed mass distributions on $\mathbb{R}^{2n}$ can be bisected by a pair of orthogonal hyperplanes. \end{cor}

\subsection{Equipartitions Modulo $k$}

	We now discuss some more general equipartition results which follow from Theorem 4. We make the following definition. 
	
\begin{df} A regular $km$-fan $F_{km}$ will be said to equipartition a signed mass distribution $\mu$ \textit{modulo k} if its corresponding sectors satisfy $\mu(\mathcal{S}_i)=\mu(\mathcal{S}_{i+k})$ for each $0\leq i<km$, that is, if each of the $k$ regular $m$-fans composing $F_{km}$ equipartitions $\mu$.\end{df}

\noindent Arguing as in the proofs of Corollaries 5 and 6 applied to the subgroups $\mathbb{Z}_p\leq \mathbb{Z}_{kp}$ and $\mathbb{Z}_4\leq \mathbb{Z}_{4k}$, one has the following two \textit{modulo} equipartition results:

\begin{cor} Let $p$ be an odd prime. Any $n$ signed mass distributions on $\mathbb{R}^{k(p-1)n}$ can be equipartitioned \textit{mod} $k$ by a regular $kp$-fan. \end{cor}

\begin{cor} Any $n$ signed mass distributions on $\mathbb{R}^{2kn}$ can be equipartitioned \textit{mod} $2k$ by a regular $4k$-fan.\end{cor}

For instance, any signed mass distribution on $\mathbb{R}^4$ can be equipartitioned (\textit{mod} 1) by a regular 5-fan, equipartitioned \textit{mod} 2 by a regular 6-fan,  and bisected by each of the 4 hyperplanes composing some regular 8-fan. For a single measure on $\mathbb{R}^{2(p-1)}$, $p$ odd, these \textit{modulo} 2 equipartitions of a single signed mass distribution by a regular $2p$-fan can be compared with the equipartition results [13] for a single measure on $\mathbb{R}^{p+1}$ mentioned above, particularly when $p=3$.\\   

	 Equipartitions by regular $m$-fans form a small subset of the more general study of equipartitions of measures in $\mathbb{R}^n$ by arbitrary $m$-fans, of which there are a great deal of results, especially in the planar case. Most relevant to the results at hand are those of B\'ar\'any and Matou\v sek, who showed that any \textit{two} absolutely measures in the plane can be equipartitioned by a 3-fan [4] and a 4-fan [3], implying (by letting one of the measures be a unit disk) that a single measure in the plane can be equipartitioned by regular 3 and 4-fans. More recently, it was shown in [2] that any convex body in the plane can be be partitioned by a convex 3-fan (i.e., one for which the angle between any two half-lines is no greater than $\pi$) into three internally disjoint convex sets of both equal area \textit{and} perimeter. These results can be compared to Corollary 5 when $n=1$ and $p=3$ and to Corollary 6 when $n=1$. Finally, we note  that \v Zivaljevi\'c and Vr\'ecica   [15] showed that any $n-1$ absolutely continuous measures on $\mathbb{R}^n$ can be trisected by a single regular 3-fan, which is an improvement of Corollary 3 when $p=3$.

\section{Quaternionic Ham Sandwich Theorems}

	As for complex multiplication by $S^1$, quaternionic multiplication by fixed elements of $S(\mathbb{H})=S^3=\{u\in\mathbb{H}\mid |u|=1\}$ affords a rotational description. Realizing the 2-sphere $S^2$ inside the purely imaginary quaternions $\{bi + cj + dk \mid b,c,d\in\mathbb{R}\}$, each $u\in S^3$ can be expressed in ``polar coordinates" by $u = e^{\theta x}:= \cos\theta + \sin\theta x$ for some $x \in S^2$, and this $x$ is unique provided $u\neq\pm1$. Letting $P:= \{a + bx \mid a,b\in\mathbb{R}\}$ denote the plane generated by 1 and $x$, it follows that multiplication on the right by $u$ (called a left screw) rotates $P$ by $\theta$ and the plane $P^\perp$ orthogonal to it by $-\theta$, while multiplying on the left by $u$ (a right screw) rotates both planes by $\theta$ (see, e.g., [8]).
	
 	The finite subgroups of $S^3$ are classified using the canonical 2-fold covering homomorphism $\varphi: S^3 \longrightarrow SO(3)$ of the special orthogonal group via conjugation. It is a very classical fact (e.g., [1]) that the finite subgroups of $SO(3)$ are either (a) cyclic, (b) dihedral $D_m$ ($m\geq2$), or (c) the rotational groups $T, O, I$ of the tetrahedron, octahedron, and icosahedron, respectively. Other than the cyclic groups of odd-order, the finite subgroups of $S^3$ are precisely the pullbacks of these subgroups under $\varphi$. Explicitly, the finite subgroups of $S^3$ are of the following form (see, e.g., [8]):
	
\begin{itemize}
\item Cyclic groups $C_m=\{\zeta_m^p\mid 0\leq p<m\}$
\item Binary Dihedral Groups $D_m^* =\varphi^{-1}(D_m)= \{\zeta_{2m}^p,\zeta_{2m}^qj\mid 0\leq p,q<2m\}$
\item The Binary Tetrahedral Group $T^*=\varphi^{-1}(T)$, the Binary Octahedral Group $O^*=\varphi^{-1}(O)$, and the Binary Icosahedral Group $I^*=\varphi^{-1}(I)$
\end{itemize}

\noindent We consider the Binary Polyhedral cases $G=D_m^*, T^*, O^*,$ and $I^*$ first. 

\subsection{$G$ Is Binary Polyhedral}
	
	For non-cyclic $G\leq S^3$, the Voronoi partition $\{R_g\}_{g\in G}$ of $\mathbb{H}$ by $G\subseteq\mathbb{H}$ may be described as follows. Let $Conv(G)\subseteq\mathbb{H}$ denote the convex hull of $G$, and let $P_G$ be its dual polytope $Conv(G)^*=\{w\in\mathbb{H}\mid \langle w, g\rangle_{\mathbb{R}}\leq 1\}$. Then $\partial P_G = \cup_{g\in G} C_g$ is a triangulation of $S^3$ into $|G|$ uniform 3-dimensional polyhedra $C_g =\{w\in P_G\mid \langle w, g\rangle_{\mathbb{R}} =1\}$ and $R_g = Cone(C_g) = \cup_{r\geq0}rC_g$ is the cone on these polyhedra $C_g$.  
	
	Each $P_G$ is found easily by using the definition of the dual. For instance, $Conv(Q_8)$ is the cross-polytope (16-cell) and its dual $P_{Q_8}$ is the four-dimensional cube (8-cell)  with boundary $\partial P_{Q_8}$ composed of 8 cubes. For $m>2$, $\partial P_{D^*_m}$ is the union of $4m$ uniform prisms with base a regular $2m$-gon ($Q_8=D_2^*$), while for the binary Platonic groups it follows that  $P_{T^*}$ is the 24-cell, with boundary composed of 24 regular tetrahedra, $P_{I^*}$ is the 120-cell, with boundary composed of 120 regular dodecahedra, and $\partial P_{O^*}$ is composed of 48 uniform truncated cubes (see, e.g., [8]). In particular, each of the six regular four-dimensional polyhedra [7] except the regular 4-simplex is realized by $Conv(G)$ or its dual $P_G$ when $G=Q_8, T^*$, and $I^*$. 
	
	Letting $(\mathbf{a},b)\in S^{4n-1}\times\mathbb{H}$, we see that each \begin{equation} \mathcal{R}_g(\mathbf{a},b)=\{ \mathbf{u}\in\mathbb{H}^n\mid \langle \mathbf{u},\mathbf{a}\rangle_{\mathbb{H}} = \bar{b} + v; v\in R_g\} \end{equation}  is a ``polyhedral wedge", the sum of the (real) codimension 4 affine space $H_{\mathbb{H}}$ and a copy of the cone $R_g$ lying in the orthogonal complement $H_{\mathbb{H}}^\perp$. The group $G$ acts freely and transitively on each collection $\{\mathcal{R}_g\}_{g\in G}$ of polyhedral wedges by the symplectic isometries which rotate $\mathbb{H}^n$ about $H_{\mathbb{H}}$ by the left screws $g\in G$: $\mathcal{R}_{g_1}(\mathbf{a},b)\cdot g_2: = \mathcal{R}_{g_1g_2}(\mathbf{a},b)$ for each $g_1,g_2\in G$. 
	
	Now consider the $(G,\varphi)$-averages $\sum_{g\in G}\varphi^{-1}(g)\mu(\mathcal{R}_g)\in\mathbb{H}$ of a quaternion-valued mass distribution $\mu$ on $\mathbb{H}^n$ and an automorphism $\varphi$ of $G$. A left $\mathbb{H}$-linear $G$-action $\cdot$ on $\mathbb{H}$ preserving the right symplectic inner product $(u,v)\mapsto \bar{u}v$ is called left symplectic, and hence the free left symplectic $G$-actions on $\mathbb{H}$ are given precisely by $g\cdot u = \varphi(g)u$ for each $g\in G$ and $u\in\mathbb{H}$ when $\varphi\in Aut(G)$. Thus the various $(G,\varphi)$-averages (2) evaluate the $G$-symmetry of the measures of the polyhedral wedges with respect to the left symplectic free $G$-actions on $\mathbb{H}$.\\
	
	We  conclude our discussion of the binary polyhedral case by giving some applications for real measures. By considering the Binary Dihedral groups $D_p^*$, $p$ prime, and applying Theorem 2 to the subgroup $C_p$, we obtain the following result akin to the \textit{modulo} 2 equipartition results of Corollary 7:
				
\begin{cor} Let $p$ be a prime number. Given any $n$ signed mass distributions $\mu_1,\ldots, \mu_n$ on $\mathbb{R}^{4(p-1)n}$, there exists a collection of $4p$ $2p$-prism wedges $\{\mathcal{R}_g\}_{g\in D_p^*}$ partitioning $\mathbb{R}^{4(p-1)n}$ such that $\mu_i(\mathcal{R}_{\zeta_{2p}^k})=\mu_i(\mathcal{R}_{\zeta_{2p}^{k+2}})$ and $\mu_i(\mathcal{R}_{\zeta_{2p}^kj})=\mu_i(\mathcal{R}_{\zeta_{2p}^{k+2}j})$ for each $0\leq k <2p$ and each  $1\leq i \leq n$.\end{cor}

Thus alternating prism wedges within each of the two ``bands"  of wedges composing $\mathbb{R}^{4(p-1)n}$  always have equal measure. Applications for $Q_8$, the binary Platonic groups $G=T^*, O^*,$ and $I^*$, and partitions by cubical, tetrahedral, truncated cubical, or regular dodecahedral wedges, respectively, are obtained by applying Theorem 1 to $Q_8$ and Theorem 2 to $Q_8\leq G$:

\begin{cor} Let $G=Q_8, T^*, O^*,$ or $I^*$. Given any $n$ signed mass distributions $\mu_1,\ldots,\mu_n$ on $\mathbb{R}^{\frac{|G|}{2}n}$, there exists a partition of $\mathbb{R}^{n\frac{|G|}{2}}$ into $|G|$ polyhedral wedges  such that $\mu_i(\mathcal{R}_g)=\mu_i(\mathcal{R}_{-g})$ for each $g\in G$ and $1\leq i\leq n$.\end{cor}

	In other words, each pair $\{\mathcal{R}_g,\mathcal{R}_{-g}\}$ of opposite polyhedral wedges has equal measure with respect to each $\mu_i$, $1\leq i \leq n$. 
	
\subsection{$G$ is Cyclic} 

	For $G=\mathbb{Z}_m$, it is immediate that that the $\mathcal{R}_{\zeta_m^k}$ are again regular $m$-sectors, though now $\mathbb{Z}_m$ rotates the sectors by left screws about a quaternioinic hyperplane. As a quaternion-valued mass distribution on $\mathbb{H}^n$  is a pair of complex-valued mass distributions on $\mathbb{C}^{2n}$, Theorem 1 yields the $\mathbb{Z}_m$-Ham Sandwich Theorem for $2n$ complex-valued mass distributions $\mu_1,\ldots, \mu_{2n}$ on $\mathbb{C}^{2n}$ and $n$ integers $r_1=r_2,\ldots, r_{2n-1}=r_{2n}$ relatively prime to $m$. Thus the $G$-Ham Sandwich Theorem for quaternion-valued mass distributions reduces to a special case of the even-dimensional $\mathbb{Z}_m$-Ham Sandwich Theorem when $G=\mathbb{Z}_m$, just as the $\mathbb{Z}_m$-Ham Sandwich Theorem reduced to the even-dimensional Ham Sandwich Theorem when $m=2$.

\section{Proofs of Theorems 1 and 2}

	Theorems 1 and 2 are group-theoretic statements about measures, and their proofs are likewise group-theoretic, with the main idea being to connect two different $G$-actions - the free $G$-action on the unit sphere $S(\mathbb{F}^{n+1})$ by scalar multiplication on the one hand, and the free and transitive $G$-action on each $G$-Voronoi partition $\{\mathcal{R}_g\}_{g\in G}$ on the other. To almost each $\mathbf{u}\in S(\mathbb{F}^{n+1})$, we assign a $G$-Voronoi-partition $\{\mathcal{R}_g(\mathbf{u})\}_{g\in G}$ in a continuous way that respects the two $G$-actions. Taking measures of these parametrized families of $G$-regions defines a continuous map, and applying an analogue of the Borsuk-Ulam Theorem for the finite group $G$ yields Theorem 1. Here are the details:

\begin{proof}

	For each $\mathbf{u}=(u_0,u_1,\ldots, u_n)\in S(\mathbb{F}^{n+1})$, define the sets 
\begin{equation} \mathcal{R}_g(\mathbf{u}) =  \{\mathbf{x}\in \mathbb{F}^n\mid \langle \mathbf{x}, (u_1,\ldots, u_n) \rangle_{\mathbb{F}} = - \bar{u}_0 + v; v\in R_g\} \end{equation} 
for each $g\in G$, where $\{R_g\}_{g\in G}$ is the Voronoi partition of $\mathbb{F}$ by $G\subseteq\mathbb{F}$. By properties of conjugation, we have : \begin{equation} \mathcal{R}_{g_1}(g_2\mathbf{u}) = \mathcal{R}_{g_1g_2}(\mathbf{u}) \end{equation} for each $g_1,g_2\in G$. 

	When $\mathbf{u}\notin S^{d-1}\times 0$, the $\mathcal{R}_g(\mathbf{u}) = \mathcal{R}_g(\frac{(u_1,\ldots,u_n)}{||(u_1,\ldots,u_n)||},\frac{-u_0}{||(u_1,\ldots,u_n)||})$ are the fundamental $G$-regions of a $G$-Voronoi partition, and  hence by (18) the association $\mathbf{u}\mapsto \{\mathcal{R}_g(\mathbf{u})\}_{g\in G}$ preserves the free and transitive $G$-action on each collection of fundamental $G$-regions.
	
	To ensure continuity of our construction, we exclude from $S(\mathbb{F}^{n+1})$ the $G$-symmetric set $X_G:=\cup_{g\in G} X_g\times 0$,  $X_g = \partial R_g \cap S(\mathbb{F})$. Thus $G$ acts freely on $S(\mathbb{F}^{n+1})-X_G$ as before. Supposing $\mathbf{u}=(u_0,0)\in S(\mathbb{F}) \times 0$ ($u_0\notin \cup_{g\in G}\partial R_g$), we have $u_0\in Int(R_g)$ for some unique $g\in G$. As conjugation in $\mathbb{F}$ is an isometry, it follows from (17) that $\mathcal{R}_g(\mathbf{u}) = \{\mathbf{x}\in\mathbb{F}^n\mid u_0\in R_{g^{-1}}\}$ so that there exists some $g_0\in G$ such that $\mathcal{R}_{g_0}(\mathbf{u}) = \mathbb{F}^n$ and $\mathcal{R}_g(\mathbf{u})=\emptyset$ for $g\neq g_0$. 
	
	Given $\mathbb{F}$-valued mass distributions $\mu_1,\ldots, \mu_n$ on $\mathbb{F}^n$, define $f=(f_1,\ldots, f_n): S(\mathbb{F}^{n+1}) - X_G \longrightarrow \mathbb{F}^n$ by  $f_i(\mathbf{u}) = \mu_i(\mathcal{R}_1(\mathbf{u}))$ for each $1\leq i \leq n$. As each  $f_i$  is continuous (Proposition 10), Theorem 11 below and (18) yield the existence of some $\mathbf{u}\in S(\mathbb{F}^{n+1}) - X_G$ for which \begin{equation} \sum_{g\in G} \varphi_i(g)^{-1}\mu_i(\mathcal{R}_g(\mathbf{u})) = \sum_{g\in G} \varphi_i(g)^{-1} \mu_i(\mathcal{R}_1(g\mathbf{u}))=\sum_{g\in G} \varphi_i(g)^{-1}f_i(g\mathbf{u})=0\end{equation} for each $1\leq i \leq n$.
	
	The proof will therefore be complete once it is shown that $\mathbf{u}\notin S(\mathbb{F}) \times0$. Assuming otherwise, the discussion above shows that $\mathcal{R}_{g_0}(\mathbf{u}) = \mathbb{F}^n$ and $\mathcal{R}_g(\mathbf{u})=\emptyset$ if $g\neq g_0$ for some $g_0\in G$, which by (19) yields $\mu_i(\mathbb{F}^n)=0$ for each $1\leq i \leq n$, contradicting the definition of a $\mathbb{F}$-valued mass distribution.\end{proof}
	
\begin{prop} Each $f_i$ is continuous. \end{prop}
 
 We show that the association $\mathbf{u}\mapsto \mu(\mathcal{R}_1(\mathbf{u}))$ is continuous for any $\mathbb{F}$-valued  mass distribution $\mu$ on $\mathbb{F}^n$. Our proof is analogous to that of a similar statement in the proof of the Ham Sandwich Theorem given by Mato\v usek [12]. More generally, our proof of Theorem 1 reduces to the proof  there when $\mathbb{F}=\mathbb{R}$. 
 
 \begin{proof} Writing $\mu= \sum_{b\in\mathcal{B}(\mathbb{F})}\mu_bb$,  let $\mu_b=\mu_b^+-\mu_b^-$ be the Jordan decomposition of each signed mass-distribution $\mu_b$ into mutually singular positive Borel measures $\mu_b^+$ and $\mu_b^-$. That is, there exist disjoint Borel sets $A_b$ and $B_b$ whose union is $\mathbb{F}^n$, such that $\mu_b^+ = 0$ on $B_b$ and $\mu_b^- = 0$ on $A_b$ (see, e.g., [9]). It is immediate that each $\mu_b^\pm$ is a mass distribution on $\mathbb{R}^{dn}$.

	To show that each assignment $\mathbf{u}\mapsto \mu_b^\pm(\mathcal{R}_1(\mathbf{u}))$ is continuous on $S^{d(n+1)-1}-X_G$, let $\{\mathbf{u}_m\}_{m=1}^{\infty}$ be a sequence in $S^{d(n+1)-1}-X_G$ converging to $\mathbf{u}$, and let $h = \chi_{\mathcal{R}_1(\mathbf{u}))}$ and $h_m = \chi_{\mathcal{R}_1(\mathbf{u}_m)}$ be the corresponding indicator functions. We  show below that the $h_m$ converge to $h$ pointwise outside of a null set, and hence almost everywhere with respect to each $\mu_b^\pm$. As each $h_m$ is dominated by   $\chi_{\mathbb{F}^n}$, we have 
\begin{equation} \lim_{m \to \infty}\mu_b^\pm(\mathcal{R}_1(\mathbf{u}_m)) = \lim_{m \to \infty} \int h_m \mathrm{d}\mu_b^\pm = \int h \mathrm{d}\mu_b^\pm = \mu_b^\pm(\mathcal{R}_1(\mathbf{u})) \end{equation} by the Dominated Convergence theorem. 
	
	It remains to show the convergence of the $h_m$ outside a null set. To this end, define $\partial\mathcal{R}_g(\mathbf{u}) = \{\mathbf{x}\in\mathbb{F}^n\mid \langle \mathbf{x}, (u_1,\ldots, u_n) \rangle_{\mathbb{F}} = - \bar{u}_0 + v; v\in \partial R_g\}$ and $\partial G(\mathbf{u})=\cup_{g\in G}\partial\mathcal{R}_g(\mathbf{u})$ for each $\mathbf{u}\in S(\mathbb{F}^{n+1})-X_G$. If $\mathbf{u}\notin S(\mathbb{F}) \times0$, then $\partial\mathcal{R}_g(\mathbf{u})$ is the boundary of the $G$-Voronoi region $\mathcal{R}_g$, which by construction is contained in a union of hyperplanes, and so $\partial G(\mathbf{u})$ a null set. Supposing $\mathbf{u}\in S(\mathbb{F})\times 0$, one has $\partial G(\mathbf{u}) = \emptyset$, so that $\partial G(\mathbf{u})$ is always a null set. 
	
	We now show that the $h_m$ converge to $h$ outside $\partial G(\mathbf{u})$. Assume $\mathbf{w}\notin \partial G(\mathbf{u})$. For each $\mathbf{y}\in S(\mathbb{F}^{n+1})-X_G$, set $Int(\mathcal{R}_g(\mathbf{y})): = \{\mathbf{x}\in\mathbb{F}^n\mid \langle \mathbf{x}, (y_1,\ldots, y_n) \rangle_{\mathbb{F}} = -\bar{y}_0 + v; v\in Int(R_g)\}$. By the definition of $\partial G(\mathbf{u})$, $\mathbf{w}\in Int(\mathcal{R}_{g_0}(\mathbf{u}))$ for some unique $g_0\in G$. Equivalently, $\mathbf{u}\in \psi^{-1}(Int(R_{g_0}))$, where $\psi: S(\mathbb{F}^{n+1})-X_G \longrightarrow \mathbb{F}$ is given by $\psi(\mathbf{y}) = \langle \mathbf{w}, (y_1,\ldots, y_n)\rangle_{\mathbb{F}} + \bar{y}_0$. Since $\mathbf{u}_m$ converges to $\mathbf{u}$ in $S(\mathbb{F}^{n+1})-X_G$, the continuity of $\psi$  shows that $\mathbf{u}_m\in\psi^{-1}(Int(R_{g_0}))$ and hence $\mathbf{w}\in Int(\mathcal{R}_{g_0}(\mathbf{u}_m))$ for all sufficiently large $m$. Thus $h_m(\mathbf{w})=h(\mathbf{w})$ for all large enough $m$ and $h_m$ converges to $h$ outside of $\partial G(\mathbf{u})$. \end{proof}
	
	We now discuss the $G$-equivariance theorem on which Theorem 1 depends. For each non-zero finite subgroup $G\leq S(\mathbb{F})$, any $n$-tuple $(\varphi_1,\ldots,\varphi_n)$ of automorphisms of $G$ defines a left free $\mathbb{F}$-linear $G$-action on $\mathbb{F}^n$ by letting $G$ act coordinate-wise by $\varphi_i$: $g\cdot \mathbf{u} = (\varphi_1(g)u_1,\ldots,\varphi_n(g)u_n)$ for each $\mathbf{u}=(u_1,\ldots,u_n) \in\mathbb{F}^n$ and $g\in G$. When each $\varphi_i$ is the identity, $G$ acts by scalar multiplication, and we denote $g\cdot\mathbf{u}$ by $g\mathbf{u}$, as usual. One has the following ``Borsuk-Ulam Theorem" for these actions:

\begin{thm} Let  $f: S(\mathbb{F}^{n+1}) - X_G \longrightarrow \mathbb{F}^n$ be a continuous function. Then there exists a $G$-orbit $\{g\mathbf{u}\mid g\in G\}\subseteq S(\mathbb{F}^{n+1}) - X_G$ such that  \begin{equation} \sum_{g\in G} g^{-1}\cdot f(g \mathbf{u}) = 0 \end{equation} \end{thm}

	As $X_{S(\mathbb{R})}=\emptyset$, Theorem 11 recovers the original Borsuk-Ulam Theorem - for any continuous function $f:S^n\longrightarrow \mathbb{R}^n$, there exists some pair of antipodal points $\{\mathbf{u},-\mathbf{u}\}\subseteq S^n$ which have the same image: $f(\mathbf{u})=f(-\mathbf{u})$.

\begin{proof} Supposing that $f'(\mathbf{x}):=\sum_{g\in G} g^{-1}\cdot f(g\mathbf{x}) \neq 0$ for each $\mathbf{x}\in S(\mathbb{F}^{n+1})-X_G$, let $k(\mathbf{x}):= \frac{f'(\mathbf{x})}{||f'(\mathbf{x})||}$, and let $h:= k\circ i: S(\mathbb{F}^n) \longrightarrow S(\mathbb{F}^n)$ be the composition of $k$ and the inclusion map $i: S(\mathbb{F}^n)\hookrightarrow S(\mathbb{F}^{n+1}) - X_G$, $\mathbf{u}\hookrightarrow (0,\mathbf{u})$. 

	It is clear that $i$ (and hence $h$) is nullhomotopic, since $0\times S(\mathbb{F}^n)$ is the boundary of the $dn$-dimensional disk $D^{dn}_v\subseteq S(\mathbb{F}^{n+1})-X_G$ formed by taking the union of all great circle arcs from a fixed $(v,0)\in S(\mathbb{F})\times 0$ to $0\times S(\mathbb{F}^n)$, provided $v\notin \cup_{g\in G} X_g$.  On the other hand, $h$ is $G$-equivariant: $h(g\mathbf{x})=g\cdot h(\mathbf{x})$ for each $g\in G$ and $\mathbf{x}\in S(\mathbb{F}^n)$. As each $G\leq S(\mathbb{F})$ contains the cyclic subgroup $C_m=\{\zeta_m^k\}_{k=0}^{m-1}$ and restricting $\varphi\in Aut(G)$ to $C_m$ is an automorphism of $C_m$, $h$ is in particular $C_m$-equivariant. It is a standard fact of algebraic topology that such a map $h$ has degree relatively prime to $m$, as follows from a study of the fundamental group and cohomology rings of the quotient manifolds $S(\mathbb{F}^n)/\mathbb{Z}_m$ associated to these actions, their covering fibrations, and the map $\bar{h}: S(\mathbb{F}^n)/\mathbb{Z}_m\rightarrow S(\mathbb{F}^n)/\mathbb{Z}_m$ induced by $h$ (see, e.g., [10]). Therefore, $h$ cannot be nullhomotopic. \end{proof}
	
	To prove Theorem 2, let $H\leq G$. Theorem 11 applied to $H$ still holds if $X_H$ is replaced by $X_G$, since $H$ acts freely on $S(\mathbb{F}^{n+1})-X_G$. Letting $g_1,\ldots, g_k$ be representatives of the cosets $g_1H,\ldots, g_kH$ of $H$, the map $f: S(\mathbb{F}^{n+1})-X_G\rightarrow \mathbb{F}^{kn}$ defined by $f_{\ell_i}(\mathbf{x})=\mu_i(\mathcal{R}_{g_\ell}(\mathbf{x}))$ for each $1\leq \ell \leq k$ and each $1\leq i \leq n$ is still continuous. Given automorphisms $\varphi_{\ell_i},\ldots, \varphi_{\ell_n}\in Aut(H)$, $1\leq \ell \leq k$, there exists therefore some $\mathbf{u}\in S(\mathbb{F}^{n+1})-X_G$ for which $\sum_{h\in H} \varphi_{\ell_i}(h)^{-1}f_{\ell_i}(h\mathbf{u})=0$ for each $1\leq \ell\leq k$ and $1\leq i\leq n$. As $\mathcal{R}_{g_\ell}(h\mathbf{u})=\mathcal{R}_{g_\ell h}(\mathbf{u})$, the same argument as before completes the proof.
	 
\subsection{Towards a More General $G$-Ham Sandwich Theorem}

 	We conclude this paper by examining a promising extension of Theorem 1 to other finite groups. In section 2, we showed that the space $\mathcal{P}(G;\mathbb{F}^n)$ of all $G$-Voronoi partitions of $\mathbb{F}^n$ is canonically identifiable with the tautological $\mathbb{F}$-line bundle $E=(S(\mathbb{F}^n)\times\mathbb{F})/G$ over the the spherical space form $S(\mathbb{F}^n)/G$ given by the standard $G$-action on $\mathbb{F}^n$. Thus these partitions are naturally associated to the manifold $S(\mathbb{F}^n)/G$, with this manifold itself realized combinatorially as those partitions whose centering $\mathbb{F}$-hyperplanes are linear. For instance, the space of all $\mathbb{Z}_m$-Voronoi partitions of $\mathbb{C}^n$ is the tautological complex line bundle over the standard lens space $L^{2n-1}(m)/\mathbb{Z}_m$, with $L^{2n-1}(m)$ itself identified as those partitions centered about complex linear hyperplanes. 
 	
	One has the same notion of $G$-Voronoi partitions for any group $G$ which acts freely and linearly on spheres (these groups are classified completely by Wolf [14]).  For instance, suppose that $\rho_0: G\rightarrow GL(V)$ is a $d$-dimensional irreducible fixed-point free representation, and extend this action diagonally to $\rho: G\rightarrow GL(V^n)$. Each translated $G$-orbit of $\mathbf{a}\in S(V^n)$ defines a Voronoi-partition of $V^n$, thereby defining a natural family $\mathcal{P}(V^n;G)$ of partitions of $V^n$ into regular convex fundamental regions $\{\mathcal{R}_g\}_{g\in G}$. As before, the $G$-Voronoi regions are centered about a special class of codimension $d$-affine subspaces, $G$ acts freely and transitively on $\{\mathcal{R}_g\}_{g\in G}$ by linear isometries about their fixed centers, and the space $\mathcal{P}(V^n; G)$ of all such partitions is canonically identifiable with the tautological $V$-bundle over the space-form $S(V^n)/G$ (see, e.g., [11]), with this manifold itself realized as those partitions centered about codimension $d$-linear subspaces. For example, the groups of Theorem 1 and their corresponding $G$-partitions arise from the fixed-point free irreducible representations $\mathbb{Z}_2\cong O(1)$, $\mathbb{Z}_m\hookrightarrow U(1)\cong SO(2)$ when $m\geq 3$, and $G\hookrightarrow Sp(1)\subseteq SO(4)$ when $G$ is binary polyhedral.
		
	As was the case for the groups of Theorem 1, we believe it should be the case that the $G$-symmetry on $V^n$ given by the representation $\rho:G\rightarrow GL(V^n)$ can be realized as a corresponding $G$-symmetry of its $V$-valued measures. Specifically, given any $V$-valued measures $\mu_1,\ldots, \mu_n$ on $V^n$ and any irreducible fixed-point free representations $\rho_1,\ldots, \rho_n: G\rightarrow GL(V)$, we conjecture the existence of a $G$-Voronoi partition $\{\mathcal{R}_g\}_{g\in G}$ of $V^n$ which $G$-equipartitions each of these measures, as realized by the vanishing of each $(G,\varphi_i)$-average: \begin{equation} \sum_{g\in G} \rho_i(g)^{-1}(\mu_i(\mathcal{R}_g))=0\end{equation} for each $1\leq i \leq n$. As was the case for Theorem 1, we expect the proof of this theorem to rely ultimately on the algebraic topology of the associated spherical manifolds involved.  
	 
\section{Acknowledgments}

	The author thanks his thesis advisor, Sylvain E. Cappell, whose guidance and suggestions were of great help in the development of this paper, much of which is contained in the author's doctoral thesis. The author would also like to thank Alfredo Hubard for useful discussions. This research was partially supported by DARPA grant HR0011-08-1-0092.

\end{document}